\documentclass[12pt]{amsart}

\usepackage{amsmath,amssymb,latexsym,soul,cite,mathrsfs}

\usepackage{color,enumitem,graphicx}
\usepackage[colorlinks=true,urlcolor=blue,
citecolor=red,linkcolor=red,linktocpage,pdfpagelabels,
bookmarksnumbered,bookmarksopen]{hyperref}
\usepackage[english]{babel}

\usepackage[left=2.9cm,right=2.9cm,top=2.8cm,bottom=2.8cm]{geometry}
\usepackage[hyperpageref]{backref}
\usepackage[colorinlistoftodos]{todonotes}
\makeatletter
\providecommand\@dotsep{5}
\def\listtodoname{List of Todos}
\def\listoftodos{\@starttoc{tdo}\listtodoname}
\makeatother

\numberwithin{equation}{section}

\newtheorem{theorem}{Theorem}[section]
\newtheorem{proposition}[theorem]{Proposition}
\newtheorem{lemma}[theorem]{Lemma}

\begin{document}

\title[Multiplicity of solutions for a class of quasilinear...]
{Multiplicity of solutions for a class of quasilinear problems involving the $1$-Laplacian operator with critical growth}
\author{Claudianor O. Alves, Anass Ourraoui and  Marcos T. O. Pimenta }
\address[Claudianor O. Alves ]
{\newline\indent Unidade Acad\^emica de Matem\'atica
\newline\indent
Universidade Federal de Campina Grande
\newline\indent
58429-970, Campina Grande - PB, Brazil
\newline\indent e-mail:coalves@mat.ufcg.edu.br}

\address[Anass Ourraoui  ]
{\newline\indent Department of Mathematics, FSO,
\newline\indent University of Mohamed I, Morocco,
\newline\indent e-mail:a.ourraoui@gmail.com }

\address[ Marcos T. O. Pimenta  ]
{\newline\indent  Departamento de Matem\'atica e Computa\c{c}\~ao,
\newline\indent  Universidade Estadual Paulista (Unesp), Faculdade de Ci\^encias e Tecnologia
\newline\indent  19060-900 - Presidente Prudente - SP, Brazil \\
\newline\indent e-mail:marcos.pimenta@unesp.br }

\pretolerance10000


\begin{abstract}
The aim of this paper is to establish two results about multiplicity of solutions to problems involving the $1-$Laplacian operator, with nonlinearities with critical growth. To be more specific, we study the following problem
$$
\left\{	\begin{array}{l}
	-	\Delta_1 u +\xi \frac{u}{|u|} =\lambda |u|^{q-2}u+|u|^{1^*-2}u, \quad\text{in }\Omega,\\
		u=0, \quad\text{on } \partial\Omega.
	\end{array}
\right.
$$
where $\Omega$ is a smooth bounded domain in $\mathbb{R}^N$, $N \geq 2$ and $\xi \in\{0,1\}$. Moreover, $\lambda > 0$, $q \in (1,1^*)$ and $1^*=\frac{N}{N-1}$. The first main result establishes the existence of many rotationally non-equivalent and nonradial solutions by assuming that $\xi=1$, $\Omega = \{x \in \mathbb{R}^N\,:\,r < |x| < r+1\}$, $N\geq 2$, $N \not = 3$ and $r > 0$. In the second one, $\Omega$ is a smooth bounded domain, $\xi=0$, and the multiplicity of solutions is proved through an abstract result which involves genus theory for functionals which are sum of a $C^1$ functional with a convex lower semicontinuous functional.
\end{abstract}

\thanks{ C. O. Alves was partially
supported by  CNPq/Brazil 304804/2017-7. M. T. O. Pimenta was partially supported by FAPESP 2019/14330-9, CNPq/Brazil 303788/2018-6 and FAPDF }
\subjclass[2010]{1-Laplacian operator; Functions of bounded variation, Varitional Methods}

\maketitle

\section{Introduction}
In this work we are concerned with the existence of multiple solutions for the following class of problem
\begin{equation} \label{P0}
	\left\{	\begin{array}{l}
	-	\Delta_1 u + \xi \frac{u}{|u|}=\lambda |u|^{q-2}u+|u|^{1^*-2}u, \quad\text{in }\Omega,\\
		u=0, \quad\text{on } \partial\Omega,
	\end{array}
	\right.
\end{equation}
where $\Omega$ is a smooth bounded domain in $\mathbb{R}^N,$ $q \in (1,1^*)$, $\xi \in \{0,1\}$, $\lambda>0$ and $1^*=\frac{N}{N-1}$ for $N \geq 2$. 

Problem (\ref{P0}) looks as the formal limit, as $p\to1^+$, of
\begin{equation} \label{Pm}
	\left\{	\begin{array}{l}
	-	\Delta_p u + \xi |u|^{p-2}u= \lambda |u|^{q-2}u+|u|^{p^*-2}u, \quad\text{in }\Omega,\\
		u=0, \quad\text{on } \partial\Omega,
	\end{array}
	\right.
\end{equation}
where $p^*=\frac{Np}{N-p}$ for $N \geq p$. 

The interest in this sort of problem has started with the celebrated paper of Br\'ezis and Nirenberg \cite{Brez}, in which the authors proved that, for $p =q= 2$ and $\xi=0$, \eqref{Pm} admits a positive solution for every $\lambda\in(0,\lambda_1)$ and $N \geq 4$. Later, this result has been extended for $p > 1$ by Egnell \cite{Egne}, Garc\'ia Azorero and Peral Alonso \cite{GarciaPeral} and Gueda and Veron \cite{GuedaVeron}.

As far as works involving the $1-$Laplacian operator are regarded, some of the pioneering works involving this operator were written by Andreu, Ballesteler, Caselles and Maz\'on in a series of papers (among them \cite{Andreu,AndreuBallesterCasellesMazon,AndreuBallesterCasellesMazonDIE}), which gave rise to the monograph \cite{AndreuCasellesMazonMonograph}. Indeed, in \cite{Andreu}, the authors characterize the imprecise quotient $\displaystyle \frac{Du}{|Du|}$ (when $Du$ is just a Radon measure, rather than an $L^1$ function), through the Pairing Theory of Anzellotti (see \cite{AnzellottiAMPA1983} and also \cite{AndreuCasellesMazonMonograph}). This theory allows them to introduce a vector field ${\textbf z} \in L^\infty(\Omega,\mathbb{R}^N)$ which plays the role of  $\displaystyle \frac{Du}{|Du|}$. Among the very first works on this issue we could also cite the works of Kawohl \cite{Kawohl} and also Demengel \cite{Demengel1}, where in the later, the author used the symmetry of the domain to get nodal solutions to problems involving the $1-$Laplacian operator and a nonlinearity with critical growth.

In \cite{DegiovanniMagrone},  Degiovanni and Magrone  studied the version of the Br\'ezis-Nirenberg problem to the 1-Laplacian operator, by applying a linking theorem. In that work, for compactness issues, they worked with an extension of the energy functional to the Lebesgue space $L^{1^*}(\Omega)$.

In \cite{FigueiredoPimenta1}, Figueiredo and Pimenta studied a problem related to (\ref{P0}), where the nonlinearity has a subcritical growth. In their main result, an approach based on the Nehari method has been developed in order to obtain ground-state solutions.

Regarding quasilinear problems of this type, the natural space to deal with it is the space of functions of bounded variation, $BV(\Omega)$. More specifically, when dealing with this sort of problems through variational methods, some difficulties related to the Palais-Smale condition arise. Moreover, other ones related to the lack of smoothness of the energy functional and to the lack of reflexiveness of $BV(\Omega)$, arise as well. 

In this work, we exploit some facts and ideas from the above papers, especially from \cite{Demengel1} to show the existence of multiple
nontrivial solutions to (\ref{P0}). Our goal is twofold. First, we establish the existence of many rotationally
non-equivalent and nonradial solutions for the above problem  (\ref{P0}) with $\xi=1$, involving a nonlinearity with
critical growth in the case when 
\begin{equation} \label{omegar}
\Omega=\Omega_r=\{x \in \mathbb{R}^N\,:\,r < |x| < r+1\}.
\end{equation}
Afterwards, we study (\ref{P0}) by assuming that $\Omega$ is a smooth bounded domain in $\mathbb{R}^N$ and $\xi=0$. In this case, we prove the existence of multiple many solutions, by applying a version of an abstract result in \cite{Szulkin}.

The existence of many rotationally non-equivalent and nonradial solutions was considered in some problems involving the Laplacian operator. In  Br\'ezis and Nirenberg \cite{Brez},  it was proved the existence of a non-radial positive solution for the following problem
\begin{equation} \label{GuidasNN}
\left\{	\begin{array}{l}
	-	\Delta u+u-u^{p}=0, \quad\text{in $D$,}\\
	u=0, \quad\text{on } \partial D,
\end{array}
\right.	
\end{equation} 	
where 
$$
D=\{x \in \mathbb{R}^N:\, r < |x| < r+d\}
$$
for some $d>0$. In \cite{Coffman}, Coffman proved that, if $p > 1$ and $N = 2$ or $1 < p < N/(N-2)$ and $N \geq  3$, the number
of nonradial and rotationally non-equivalent positive solutions of (\ref{GuidasNN}), tends
to $+\infty$ as $r \to +\infty$.

Motivated by the above papers, some authors have studied this class of problems.
For the subcritical case, we can mention the papers of  Li \cite{Li}, Lin \cite{Lin}, Suzuki \cite{Suzuki} and references therein. Related to the critical case, Wang and Willem \cite{WangWillem} proved the
existence of multiple solutions for the following problem 
\begin{equation} \label{WW}
	\left\{	\begin{array}{l}
		-	\Delta u=\lambda u+u^{2^*-1}, \quad\text{in } \Omega_r,\\
		u=0, \quad\text{on } \partial \Omega_r,
	\end{array}
	\right.	
\end{equation} 
where $\Omega_r$ is given in (\ref{omegar}). The authors proved that
for $0 < \lambda < \pi^2$ and $n \in \mathbb{N}$, there exists $R(\lambda, n)$ such that for $r > R(\lambda, n)$, (\ref{WW}) has at least $n$ nonradial and rotationally non-equivalent solutions.
Inspired by \cite{WangWillem}, de Figueiredo and Miyagaki \cite{DjairoOlimpio} considered
the following problem
\begin{equation} \label{FO}
	\left\{	\begin{array}{l}
		-	\Delta u=f(|x|,u)+u^{2^*-1}, \quad\text{in } \Omega_r,\\
		u=0, \quad\text{on } \partial \Omega_r,
	\end{array}
	\right.	
\end{equation} 
where $f$ is a $C^1$ function with subcritical growth.

Still related to this class of problem, we would like to cite the papers of Alves and de Freitas  \cite{AlvesFreitas}, 
Byeon \cite{Byeon}, Castro and Finan \cite{CAstroFinan}, Catrina and Wang \cite{CatrinaWang},
Mizoguchi and Suzuki \cite{MIzoguchiSuzuki}, Hirano and Mizoguchi \cite{HiranoMizoguchi} and references
therein.

Motivated by the works previously mentioned and more precisely, by \cite{DjairoOlimpio}, \cite{DegiovanniMagrone} and \cite{WangWillem}, our first main result is the following.

\begin{theorem} \label{T1}
	For each $n\in \mathbb{N}$ there is $r_0>0$ and $\lambda_0>0$ such that for all $\lambda \geq\lambda_0$ and $r\geq r_0$, \eqref{P0} for $\xi=1$ has at least $n$ nonradial and rotationally 	non-equivalent solutions.
\end{theorem}

In what follows, according to Degiovanni and Magrone \cite{DegiovanniMagrone} and  Kawohl and Schuricht \cite{Buttazzo}, we say that $u \in BV(\Omega_r)$ is a solution to (\ref{P1}) if
there are $z \in L^\infty(\Omega_r,\mathbb{R}^N)$ and $\gamma \in L^\infty(\Omega_r,\mathbb{R})$ such that

\begin{equation}
	\left\{
	\begin{array}{l}
		|z|_\infty \leq 1, \quad \mbox{div} z \in L^{N}(\Omega_r), \quad -\displaystyle \int_{\Omega_r}u\mbox{div}z dx = \int_{\Omega_r}|Du|+\int_{\partial\Omega_r}|u|\,d \mathcal{H}^{N-1},\\
		\mbox{}\\
		|\gamma|_\infty \leq 1, \quad  \gamma|u| = u \quad \mbox{a.e. in} \quad \Omega_r,\\
		\mbox{} \\
		-\mbox{div} z + \gamma = \lambda |u|^{q-2}u+|u|^{1^*-2}u, \quad \mbox{a.e. in} \quad \Omega_r.
	\end{array}
	\right.
	\label{eulerlagrangeequation}
\end{equation}

Our second main result was motivated by the study made in Wei and Wu \cite{WeiWu}, where the authors showed the existence of multiple solution for the following class of problems involving the $p$-Laplacian operator
\begin{equation} \label{WeiWu}
	\left\{	\begin{array}{l}
		-	\Delta_p u = f(x,u)+\lambda |u|^{p^*-2}u, \quad\text{in }\Omega,\\
		u=0, \quad\text{on } \partial\Omega,
	\end{array}
	\right.
\end{equation}
where $\Omega$ is a smooth bounded domain, $\lambda$ is a positive parameter and $f$ is continuous, with subcritical growth. In this work, the authors used a version of an abstract theorem due to Ambrosetti and Rabinowitz \cite{AmbrosettiRabinowitz} which involves the genus theory for $C^{1}$ even functionals. Their main result proves that given $n \in \mathbb{N}$, there is $\lambda_*=\lambda_*(n)>0$ such that problem (\ref{WeiWu}) has at least $n$ nontrivial solutions for $\lambda \in (0, \lambda_*)$. In \cite{SilvaXavier}, Silva and Xavier improved the main results proved in \cite{WeiWu}.

Our main second result has the following statement.
\begin{theorem} 
\label{T2}
Given $n \in\mathbb{N}$, there is $\lambda_n>0$ such that \eqref{P0} for $\xi=0$  has at least $n$ nontrivial solutions for $\lambda \geq \lambda_n$.	
\end{theorem}

We would like to point out that in the proof of Theorems \ref{T1} and \ref{T2} we cannot use the classical variational methods for $C^1$ functionals, since problems involving the 1-Laplacian operator have energy functionals which are not $C^1$. This, in turn, brings a lot of difficulties for deal with the problem. In order to overcome this difficulty, we use the minimax methods developed by Szulkin in\cite{Szulkin}, which works well for functionals that can be written as the sum of a $C^1$ functional with a convex lower semicontinuous one. Finally, we would like to point out that $\xi=1$ in Theorem \ref{T1} is very important, because in our approach it was necessary to work with some sequences in $BV(\mathbb{R}^N)$ (see the proof of Lemma \ref{Estimativa2}). 

Before concluding this introduction, for those readers interested in problems involving the 1-Laplacian operator, we would like to cite Alves \cite{Alves0,Alves1},  Alves and Pimenta \cite{AlvesPimenta}, Alves, Figueiredo and Pimenta \cite{AlvesFigueiredoPimenta}, Bellettini, Caselles and Novaga \cite{BCN}, Chang \cite{Chang2}, Demengel \cite{Demengel2}, Figueiredo and  Pimenta \cite{FigueiredoPimenta3, FigueiredoPimenta4},   Mercaldo,  Rossi, Segura de Le\'on and Trombetti \cite{MRST}, Mercaldo, Segura de Le\'on and Trombetti \cite{MST}, Molino Salas and Segura de Le\'on \cite{sergio}, Ortiz Chata and Pimenta \cite{Pimenta2}.

\section{Existence of nonradial solutions}

In this section, we prove Theorem \ref{T1}, which establishes the existence of many rotationally
non-equivalent and nonradial solutions for the following problem
\begin{equation} \label{P1}
	\left\{	\begin{array}{l}
	-	\Delta_1 u +\frac{u}{|u|}=\lambda |u|^{q-2}u+|u|^{1^*-2}u, \quad\text{in }\Omega_r,\\
		u=0, \quad\text{on } \partial\Omega_r,
	\end{array}
	\right.
\end{equation}
where $\Omega_r=\{x \in \mathbb{R}^N\,:\,r < |x| < r+1\}, N\geq 2$ and $N\not=3$, $r>0, \lambda>0$, $q \in (1,1^*)$ and $1^*=\frac{N}{N-1}$.

Associated with problem (\ref{P1}), we have the energy functional  $I_\lambda:BV(\Omega_r) \to \mathbb{R}$ given by
$$
I_\lambda(u)=\int_{\Omega_r} |D u|\,dx+\int_{\Omega_r}|u|\,dx+\int_{\partial\Omega_r}|u| \,d\mathcal{H}^{N-1}-\frac{\lambda}{q}\int_{\Omega_r}|u|^q\,dx-\frac{1}{1^*}\int_{\Omega_r}|u|^{1^*}\,dx.
$$

In the sequel, we say that $u \in BV(\Omega_r)$ is a solution of (\ref{P1}) if $0 \in \partial I_\lambda(u)$, where $\partial I_\lambda(u)$ denotes the generalized gradient of $I_\lambda$ in $u$, as defined in \cite{Chang}. It is possible to prove that $0 \in \partial I_\lambda(u)$ if, and only if,
\begin{equation}
	\|w\|_r- \|u\|_r \geq \int_{\Omega_r}(\lambda|u|^{q-2}u+|u|^{1^*-2}u)(w - u)dx, \quad \forall w\in BV(\Omega_r),
	\label{eqsolution}
\end{equation}
where $BV(\Omega_r)$ denotes the space of functions of bounded variation. 

We say that $u \in BV(\Omega_r)$, or $u$ is a function of bounded variation, if $u \in L^1(\Omega_r)$ and its distributional derivative $Du$ is a vectorial Radon measure, i.e.,
$$
BV(\Omega_r) = \left\{u \in L^1(\Omega_r); \, Du \in \mathcal{M}(\Omega_r,\mathbb{R}^N)\right\}.
$$
It can be proved that $u \in BV(\Omega_r)$ is equivalent to $u \in L^1(\Omega_r)$ and
$$
\int_{\Omega_r} |Du| := \sup\left\{\int_{\Omega_r} u \mbox{div}\phi dx; \, \, \phi \in C^1_c(\Omega_r,\mathbb{R}^N), \, \mbox{s.t.} \, \, |\phi|_\infty \leq 1\right\} < +\infty.
$$
The space $BV(\Omega_r)$ is a Banach space endowed with the norm
\begin{equation} \label{norma1}
\|u\|:= \int_{\Omega_r} |Du| + |u|_{L^{1}(\Omega_r)}.
\end{equation}
Moreover, the Sobolev embeddings hold also for this space and its embedding into $L^r(\Omega)$ is continuous for all $\displaystyle r \in \left[1,1^*\right]$ and compact for  $r \in [1,1^*)$.

In this section, we will consider the following norm on $BV(\Omega_r)$,
$$
\|u\|_r=\int_{\Omega_r}|Du|\,dx+\int_{\Omega_r}|u|\,dx+\int_{\partial\Omega_r}|u|\,d \mathcal{H}^{N-1},
$$
which is equivalent to the norm (\ref{norma1}), where $\mathcal{H}^{N-1}$ denotes the $(N-1)$-dimensional Hausdorff measure.

As one can see in \cite{Buttazzo}, the space $BV(\Omega_r)$ has different convergence and density properties when compared with the usual Sobolev spaces. For instance, $C^\infty(\overline{\Omega_r})$ is not dense in $BV(\Omega_r)$ with respect to the strong convergence. However, there is a weaker sense of convergence in $BV(\Omega_r)$, called {\it intermediate convergence} (or {\it strict convergence}), which makes $C^\infty(\overline{\Omega_r})$ dense on it. We say that $(u_n) \subset BV(\Omega_r)$ converges to $u \in BV(\Omega_r)$ in the sense of the intermediate convergence if
$$
u_n \to u \quad \mbox{in} \quad L^1(\Omega_r)
$$
and
$$
\int_{\Omega_r}|Du_n| \to \int_{\Omega_r}|Du|. 
$$

In what follows, $O(N)$ denotes the group of $N \times N$ orthogonal matrices. For
any integer $k \geq 1$, let us consider the finite rotational subgroup $O_k$ of $O(2)$ given
by
$$
O_k=\left\{g\in O_2:~~g(x)=\Big( x_1 \cos {\frac{2\pi l}{k}} +x_2 \sin{ \frac{2\pi l}{k}},-x_1 \sin {\frac{2\pi l}{k}}+x_2 \cos {\frac{2\pi l} {k}}\Big)\right\}.
$$
where $x=(x_1,x_2) \in \mathbb{R}^2$ and $l \in \{0,1,...,k-1\}$. We also consider  the subgroups of $O(N)$
$$
G_k=O_k\times O(N-2),~~1\leq k<\infty
$$
and
$$
G_\infty=O(N).
$$

Now related to the above subgroups, we set the subspaces
$$
BV_{G_k}(\Omega_r)=\left\{u\in BV(\Omega_r):~~u(x)=u(g^{-1}(x)),~~for~~all~~g\in G_k\right\},
$$
endowed with the norm $\|\cdot \|_r$.

From the compact embedding involving the space $BV(\Omega_r)$, it follows that the embedding
\begin{equation} \label{EMB1}
BV_{G_k}(\Omega_r) \hookrightarrow L^{t}(\Omega_r), \quad t \in [1,1^*)
\end{equation}
is compact for $1\leq k < +\infty$ and 
\begin{equation} \label{EMB1*}
	BV_{G_\infty}(\Omega_r) \hookrightarrow L^{t}(\Omega_r), \quad t \in [1,+\infty)
\end{equation}
is compact, see \cite[Lemma 2.1]{AlvesTahir}.

 Moreover, Figueiredo and Pimenta \cite{FigueiredoPimenta2} proves that the embedding
\begin{equation} \label{EMB2}
BV_{G_\infty}(\mathbb{R}^N) \hookrightarrow L^{t}(\mathbb{R}^N), \quad t \in (1,1^*)
\end{equation}
is compact as well.

In the sequel, for each $1 \leq k \leq \infty$, $J_{\lambda,k,r}$ denotes the following real numbers
$$
J_{\lambda,k,r}=\inf_{\mathcal{N}_{k,r}}I_{\lambda},
$$
with
$$
\mathcal{N}_{k,r}=\{BV_{G_k}(\Omega_r) \setminus \{0\} ;~~E_\lambda(u)=0\},
$$
where 
\begin{equation} \label{Elambda}
E_\lambda(u)=\int_{\Omega_r} |D u|\,dx+\int_{\Omega_r}|u|\,dx+\int_{\partial\Omega_r}|u| \,d\mathcal{H}^{N-1}-{\lambda}\int_{\Omega_r}|u|^q\,dx-\int_{\Omega_r}|u|^{1^*}\,dx.
\end{equation} 

The set $\mathcal{N}_{k,r}$ is called Nehari set associated with $I_\lambda$ (see \cite{FigueiredoPimenta1} for a detailed description of this set). It is possible to prove that $J_{\lambda,k,r}$ is the mountain pass levels associated with $I_\lambda$ on $BV_{G_k}(\Omega_r)$ and $BV_{G_\infty}(\Omega_r)$, respectively. Hence, there is a $(PS)$ sequence $(u_n)$ associated to $J_{\lambda,k,r}$, i.e.,
\begin{equation} \label{MP01}
I_\lambda(u_n) \to  J_{\lambda,k,r}
\end{equation}
and
\begin{equation*} \label{MP02}
\|v\|_r-\|u_n\|_r \geq \int_{\Omega_r}(\lambda |u_n|^{q-2}u_n+|u_n|^{1^*-2}u)(v-u_n)\,dx-\tau_n \|v-u_n\|_r, \quad \forall v \in  BV_{G_k}(\Omega_r).
\end{equation*}
The last inequality implies that
\begin{equation} \label{MP3}
\|u_n\|_r=\int_{\Omega_r}(\lambda |u_n|^{q}+|u_n|^{1^*})\,dx+o_n(1)\|u_n\|_r.
\end{equation}

\begin{lemma} \label{boundedness} The sequence $(u_n)$ is bounded.
	
\end{lemma}	
\begin{proof} From (\ref{MP01}) and (\ref{MP3}),
$$
J_{\lambda,k,r}+o_n(1)=I_\lambda(u_n)=I_\lambda(u_n)-\frac{1}{q}\|u_n\|_r+\frac{1}{q}\int_{\Omega_r}(\lambda |u_n|^{q}+|u_n|^{1^*})\,dx+o_n(1)\|u_n\|_r.
$$	
Then
$$
J_{\lambda,k,r}+o_n(1) \geq \frac{(q-1)}{q}\|u_n\|_r-o_n(1)\|u_n\|_r\geq \frac{(q-1)}{2q}\|u_n\|_r
$$
for $n$ large enough, showing the boundedness of the sequence.
\end{proof}	

\begin{lemma} \label{eta}
For each $\lambda>0$ fixed, there is $\eta=\eta(\lambda)>0$ that is independent of $k$ and $r>0$ such that $J_{\lambda,k,r}\geq \eta $ for $1\leq k\leq \infty$.
\end{lemma}

\begin{proof}
For each $u \in \mathcal{N}_{k,r}$ we have that
\begin{equation} \label{EQ1}
\|u\|_r=\lambda \int_{\Omega_r}|u|^q\,dx+ \int_{\Omega_r}|u|^{1^*}\,dx.
\end{equation}
In what follows, we define the function $\tilde{u}:\mathbb{R}^N \to \mathbb{R}$ given by
\begin{equation} \label{tilde}
\tilde{u}(x)=
\left\{
\begin{array}{l}
u(x), \quad x \in \Omega_r, \\
\mbox{}\\
0, \quad x \in \Omega_r^c.
\end{array}
\right.
\end{equation}
Hence,
$$
\int_{\Omega_r}|u|^q\,dx=\int_{\mathbb{R}^N}|\tilde{u}|^q\,dx, \quad \int_{\Omega_r}|{u}|^{1^*}\,dx=\int_{\mathbb{R}^N}|\tilde{u}|^{1^*}\,dx
$$
and by properties of $BV(\mathbb{R}^N)$, $\tilde{u} \in BV(\mathbb{R}^N)$ and
$$
 \int_{\mathbb{R}^N}|D\tilde{u}| = \int_{\Omega_r} |D u|\,dx+\int_{\partial\Omega_r}|u| \,d\mathcal{H}^{N-1}.
$$
The definition of $\tilde{u}$ combined with (\ref{EQ1}) gives
\begin{equation} \label{EQ2}
 \int_{\mathbb{R}^N}|D\tilde{u}|+ \int_{\mathbb{R}^N}|\tilde{u}|\,dx = \lambda \int_{\mathbb{R}^N}|\tilde{u}|^q\,dx+ \int_{\mathbb{R}^N}|\tilde{u}|^{1^*}\,dx.
\end{equation}
The function $\|\cdot\|:BV(\mathbb{R}^N) \to \mathbb{R}$ given by
\begin{equation} \label{EQ3}
\|w\|= \int_{\mathbb{R}^N}|D w|+ |w|_{L^{1}(\mathbb{R}^N)}.
\end{equation}
is a norm in $BV(\mathbb{R}^N)$. Moreover, there are positive constants $C_1,C_2>0$ such that
\begin{equation} \label{EQ4}
\|w\|_{L^{q}(\mathbb{R}^N)}\leq C_1\|w\| \quad \mbox{and} \quad \|w\|_{L^{1^*}(\mathbb{R}^N)}\leq C_2\|w\|, \quad \forall w \in BV(\mathbb{R}^N).
\end{equation}
From (\ref{EQ2})-(\ref{EQ4}), there is $C_3>0$ such that
$$
1 \leq C_3(\lambda\|\tilde{u}\|^{q-1}+\|\tilde{u}\|^{1^*-1}).
$$
Therefore, there is $\eta_1=\eta_1(\lambda)>0$ such that
$$
\|\tilde{u}\| \geq \eta_1,
$$
and so,
$$
\|u\|_r \geq \eta_1, \quad \forall u \in \mathcal{N}_{k,r}.
$$
From (\ref{EQ1}),
$$
\lambda\int_{\mathbb{R}^N}|u|^q\,dx+\int_{\mathbb{R}^N}|u|^{1^*}\,dx \geq \eta_1, \quad \forall u \in \mathcal{N}_{k,r},
$$
then,
$$
I_\lambda(u)=\frac{\lambda(q-1)}{q}\int_{\Omega_r}|u|^q\,dx+\frac{(1^*-q)}{q1^*}\int_{\Omega_r}|u|^{1^*}\,dx\ \geq \frac{(q-1)}{q}\eta_1,
$$
showing the result.
\end{proof}

Hereafter, $S$ denotes the following constant
\begin{equation} \label{S}
	S=
	\inf_{ {\footnotesize{
				\begin{array}{l}
					u \in BV(\mathbb{R}^N) \\
					u \not=0
	\end{array}}}}
	\frac{ \int_{\mathbb{R}^N}|D u|}{|u|_{{L^{1^*}(\mathbb{R}^N)}.}}.
\end{equation}

\begin{lemma} \label{Estimativa1}
For each  $1 \leq k <\infty$, there is $\lambda_k^* > 0$  such that
$$
J_{\lambda,k,r}<\frac{1}{2N}S^N, \quad \mbox{for all} \quad \lambda \geq \lambda_k^*.
$$
\end{lemma}
\begin{proof}
	From the fact that $\Omega_r$ is an open bounded domain, we may choose $\sigma>0$ that is independent of $r$, such that the ball $B_\sigma={B}_\sigma\Big(\frac{2r+1}{2},0,...,0\Big)\subset \Omega_r$ verifies
	$$ g^i {B}_\sigma\cap g^j {B}_\sigma=\emptyset,~~\mbox{for} ~~g^i\in G_k,~~i \neq j,~i,j=   0,...,k-1.$$
	
	Let us choose $ \omega\in C_0^{\infty}({B}_\sigma)\setminus\{0 \} $ and define $v:=\Sigma_{g\in G_k}g\omega\in~ {BV}_{G_k}(\Omega_r)\backslash\{0\}.$ A simple computation implies that
	$$
	I'_\lambda(tv)tv>0 \quad \mbox{for} \quad t \approx 0^+ \quad \mbox{and} \quad I'_\lambda(tv)tv\to -\infty, \quad \mbox{as} \quad t\to\infty.
	$$
Hence, there exists $t_v>0$ such that $t_v v\in \mathcal{N}_{k,r}.$ From this,
$$
J_{\lambda,k,r}\leq I_{\lambda}(t_v v) \leq kI_{\lambda}(t_v\omega)\leq k\max_{t\geq0}I_{\lambda}(t\omega).
$$
and so,	
$$
 J_{\lambda,k,r}\leq k\max_{t \geq 0}\Big\{ t\|\omega\|_r-\frac{\lambda t^q}{q}\int_{{B}_\sigma}|\omega|^q\,dx\Big\}.
$$
Putting $g(t)=t\|\omega\|_r-\frac{\lambda t^q}{q} |\omega|_q^{q},$ this function attains its maximum at
$$
t_0=\Big(\frac{\|\omega\|_r}{\lambda |\omega|_{q}^{q}}\Big)^{\frac{1}{q-1}}.
$$
Therefore,
$$
J_{\lambda,k,r}\leq k\frac{(q-1)}{q}\left(\frac{\|\omega\|_r}{|\omega|_q}\right)^{\frac{q}{q-1}}\lambda^{\frac{1}{1-q}}.
$$
Taking $\lambda_k^*> \left[\frac{2Nk(q-1)}{Sq}\right]^{q-1}\left(\frac{\|\omega\|_r}{|\omega|_q}\right)^q$ the proof is achieved.
\end{proof}



\begin{lemma} \label{Estimativa0}
For each $1\leq k \leq \infty$ the number  $J_{\lambda,k,r}$ is attained for $\lambda\geq\lambda_k^*$.
\end{lemma}

\begin{proof} The case $k=\infty$ is immediate because of the compact embedding (\ref{EMB1*}), then we will only show the case $1\leq k<+\infty$. Recalling that $J_{\lambda,k,r}$ is the mountain pass level of $I_\lambda$ on the space $BV_{G_k}(\Omega_r)$, we know that there is a $(PS)$ sequence $(u_n)$ (see (\ref{MP01})), such that,
\begin{equation*}
	I_\lambda(u_n) \to  J_{\lambda,k,r}
\end{equation*}
and
\begin{equation*}
	\|v\|_r-\|u_n\|_r \geq \int_{\Omega_r}(\lambda |u_n|^{q-2}u_n+|u_n|^{1^*-2}u)(v-u_n)\,dx-\tau_n \|v-u_n\|_r, \quad \forall v \in  BV_{G_k}(\Omega_r),
\end{equation*}
where $\tau_n \to 0$. Moreover, we also have the equality below
\begin{equation} \label{MP03}
	\|u_n\|_r=\int_{\Omega_r}(\lambda |u_n|^{q}+|u_n|^{1^*})\,dx+o_n(1)\|u_n\|_r.
\end{equation}	
Since that $(u_n)$ is bounded in $BV_{G_k}(\Omega_r)$, for some subsequence, there is $u \in BV_{G_k}(\Omega_r)$ such that
$$
\|u\|_r \leq \liminf_{n \to +\infty}\|u_n\|_r
$$
and
$$
u_n \to u \quad \mbox{in} \quad L^{t}(\Omega_r), \quad \forall t \in[1,1^*).
$$
We claim that $u \not= 0$, otherwise if we argue as in the proof of Lemma \ref{eta}, we would find a constant $C>0$ such that
\begin{equation} \label{EQ07}
	1 \leq C(\lambda\|{u}_n\|_r^{q-1}+\|{u}_n\|_r^{1^*-1}), \quad \forall n \in \mathbb{N}.
\end{equation}
On the other hand, by (\ref{MP03}),
\begin{equation*}
	\int_{\Omega_r}|D {u}_n|+ \int_{\Omega_r}|u_n|\,dx+\int_{\partial\Omega_r}|u_n| \,d\mathcal{H}^{N-1}= \int_{\Omega_r}|{u}_n|^{1^*}\,dx+o_n(1),
\end{equation*}
then for some subsequence,
$$
\lim_{n \to +\infty}\left(\int_{\Omega_r}|D {u}_n|+ \int_{\partial\Omega_r}|u_n| \,d\mathcal{H}^{N-1}\right)=\lim_{n \to +\infty}\int_{\Omega_r}|{u}_n|^{1^*}\,dx=L \geq 0.
$$
We claim that $L>0$, because otherwise we would have
$$
\lim_{n \to +\infty}\int_{\Omega_r}|{u}_n|^{1^*}\,dx=0,
$$
and so,
$$
\lim_{n \to +\infty}\|{u}_n\|_r=0,
$$
which contradicts (\ref{EQ07}). Since $L>0$, we can assume that $u_n \not=0$ for all $n \in\mathbb{N}$. Therefore, by definition of $S$, see (\ref{S}),
$$
S \leq \frac{\int_{\mathbb{R}^N}|D \tilde{u}_n|}{|\tilde{u}_n|_{{L^{1^*}(\mathbb{R}^N)}}}= \frac{\int_{\Omega_r}|D {u}_n|+ \int_{\partial\Omega_r}|u_n| \,d\mathcal{H}^{N-1}}{|{u}_n|_{{L^{1^*}(\mathbb{R}^N)}}}, \quad \forall n \in \mathbb{N},
$$
where $\tilde{u}_n$ is defined as in (\ref{tilde}). 
Letting $n \to +\infty$, we get
$$
S \leq \frac{L}{L^{\frac{1}{1^*}}}=L^{\frac{1}{N}},
$$
that is,
$$
L \geq S^{N}.
$$
Now, using the fact that
$$
J_{\lambda,k,r}+o_n(1)= I_{\lambda}({u}_n)=\frac{1}{N}\int_{\mathbb{R}^N}|\tilde{u}_n|^{1^*}\,dx+o_n(1)=\frac{1}{N}L\geq \frac{1}{N}S^N,
$$
which contradicts Lemma \ref{Estimativa1}. Now, arguing as in \cite{DegiovanniMagrone}, we also derive that
$$
\int_{\Omega_r} |D u|\,dx+\int_{\Omega_r}|u|\,dx+\int_{\partial\Omega_r}|u| \,d\mathcal{H}^{N-1}=\int_{\Omega_r}|u|^q\,dx+\int_{\Omega_r}|u|^{1^*}\,dx,
$$
from where it follows that $u \in \mathcal{N}_{k,r}$, and so,
$$
\begin{array}{l}
J_{\lambda,k,r} \leq I_\lambda(u)=\displaystyle \frac{\lambda(q-1)}{q}\displaystyle \int_{\Omega_r}|u|^q\,dx+\frac{(1^*-1)}{1^*}\int_{\Omega_r}|u|^{1^*}\,dx \\
\mbox{} \\
\hspace{0.8 cm} \leq \displaystyle \frac{\lambda(q-1)}{q}\displaystyle \lim_{n \to +\infty}\int_{\Omega_r}|u_n|^q\,dx+\liminf_{n \to +\infty}\frac{(1^*-1)}{1^*}\int_{\Omega_r}|u_n|^{1^*}\,dx\\
\mbox{} \\
\hspace{0.8 cm} \leq \displaystyle \frac{\lambda(q-1)}{q}\displaystyle \lim_{n \to +\infty}\int_{\Omega_r}|u_n|^q\,dx+\limsup_{n \to +\infty}\frac{(1^*-1)}{1^*}\int_{\Omega_r}|u_n|^{1^*}\,dx\\
\mbox{} \\
\hspace{0.8 cm} \leq \displaystyle \limsup_{n \to +\infty}(I(u_n)+o_n(1))=\displaystyle \limsup_{n \to +\infty}I(u_n)=J_{\lambda,k,r},
\end{array}
$$
from where it follows that $u \in \mathcal{N}_{k,r}$ and $I_\lambda(u) = J_{\lambda,k,r}$, showing the lemma.
\end{proof}

\begin{lemma} \label{Estimativa2}
	There exists $r_0=r_0(\lambda)>0$ such that
$$
J_{\lambda,\infty,r} \geq \frac{1}{2N}S^N,~\mbox{for}~~r> r_0.
$$
\end{lemma}
\begin{proof}
Let us assume the opposite, i.e., that there exists a sequence $r_n\to\infty$ such that
$$
J_{\lambda,\infty,r_n}<\frac{1}{2N}S^N,~\mbox{for}~~n\in \mathbb{N}. 
$$
By Lemma \ref{Estimativa0}, $J_{\infty,r_n}$ is attained for all $n\in \mathbb{N}$, and so, there is $(u_n) \in BV_{G_\infty}(\Omega_{r_n})\setminus\{0\}$ such that
$$
E_\lambda(u_n)=0 \quad \mbox{and} \quad I_\lambda(u_n)=J_{\lambda,\infty,r_n},
$$
where $E_\lambda$ is given in (\ref{Elambda}). 

The assumption  $J_{\lambda,\infty,r_n} < \frac{1}{2N}S^N$ combined with the first equality above ensures that there is $M>0$ such that
\begin{equation} \label{Boud1}
\int_{\Omega_{r_n}} |D u_n|\,dx+\int_{\Omega_{r_n}}|u_n|\,dx+\int_{\partial\Omega_{r_n}}|u_n| \,d\mathcal{H}^{N-1} \leq M, \quad \forall n \in \mathbb{N}.
\end{equation}
Setting
$$
\tilde{u}_n(x)=
\left\{
\begin{array}{l}
	u_n(x), \quad x \in \Omega_{r_n}, \\
	\mbox{}\\
	0, \quad x \in \Omega_{r_n}^c,
\end{array}
\right.
$$
we have that $\tilde{u}_n \in BV_{G_\infty}(\mathbb{R}^N)$,
$$
\int_{\Omega_{r_n}}|u_n|^q\,dx=\int_{\mathbb{R}^N}|\tilde{u}_n|^q\,dx,  \quad \int_{\Omega_{r_n}}|{u_n}|^{1^*}\,dx=\int_{\mathbb{R}^N}|\tilde{u}_n|^{1^*}\,dx
$$
and
$$
\|\tilde{u}_n\| = \int_{\Omega_{r_n}} |D u_n|\,dx+\int_{\Omega_{r_n}}|u_n|\,dx+\int_{\partial\Omega_{r_n}}|u_n| \,d\mathcal{H}^{N-1}.
$$
The definition of $\tilde{u}_n$ combined with the fact that $E_\lambda(u_n)=0$ gives
\begin{equation} \label{EQ5}
\|\tilde{u}_n\| \leq \lambda \int_{\mathbb{R}^N}|\tilde{u}_n|^q\,dx+ \int_{\mathbb{R}^N}|\tilde{u}_n|^{1^*}\,dx.
\end{equation}
From the definition of $\tilde{u}_n$, it follows that $\tilde{u}_n\to 0$ a.e in $\mathbb{R}^N$, and so, 
	\begin{equation}\label{sv}
		\tilde{u_n}\to 0~~ in ~~L^t(\mathbb{R}^N)~~ for ~~t\in(1,1^*).
	\end{equation}
Moreover, from (\ref{EQ5}), there is $t_n \in (0,1]$ such that 	
\begin{equation} \label{EQ6}
\int_{\mathbb{R}^N}|D\tilde{u}_n|+ \int_{\mathbb{R}^N}|\tilde{u}_n|\,dx=\lambda t_n^{q-1} \int_{\mathbb{R}^N}|\tilde{u}_n|^q\,dx+ t_n^{1^*-1}\int_{\mathbb{R}^N}|\tilde{u}_n|^{1^*}\,dx.
\end{equation}
Arguing as in the proof of Lemma \ref{eta}, there is $C>0$ such that
\begin{equation} \label{EQ7}
1 \leq C(\lambda\|t_n\tilde{u}_n\|^{q-1}+\|t_n\tilde{u}_n\|^{1^*-1}), \quad \forall n \in \mathbb{N}.
\end{equation}
On the other hand, by (\ref{sv}) and(\ref{EQ6}),
\begin{equation*}
\int_{\mathbb{R}^N}|t_n D\tilde{u}_n|+ \int_{\mathbb{R}^N}|t_n\tilde{u}_n|\,dx= \int_{\mathbb{R}^N}|t_n\tilde{u}_n|^{1^*}\,dx+o_n(1),
\end{equation*}
then for some subsequence,
$$
\lim_{n \to +\infty}\|t_n \tilde{u}_n\|=\lim_{n \to +\infty}\int_{\mathbb{R}^N}|t_n\tilde{u}_n|^{1^*}\,dx=L \geq 0.
$$
We claim that $L>0$, since otherwise we would have
$$
\lim_{n \to +\infty}\int_{\mathbb{R}^N}|t_n\tilde{u}_n|^{1^*}\,dx=0,
$$
and so,
$$
\lim_{n \to +\infty}\|t_n \tilde{u}_n\|=0,
$$
which contradicts (\ref{EQ7}). Since $L>0$, we can assume that $u_n \not=0$ for all $n \in\mathbb{N}$. Therefore, by definition of $S$, see (\ref{S}),
$$
S \leq \frac{\int_{\mathbb{R}^N}|D \tilde{u}_n|}{|\tilde{u}_n|_{1^*}}\leq \frac{\|\tilde{u}_n\|}{|\tilde{u}_n|_{1^*}}, \quad \forall n \in \mathbb{N}.
$$
Letting $n \to +\infty$, we get
$$
S \leq \frac{L}{L^{\frac{1}{1^*}}}=L^{\frac{1}{N}},
$$
that is,
$$
L \geq S^{N}.
$$
Now, using the fact that
$$
\frac{1}{N}S^N \leq \frac{1}{N}L= I_{\lambda}(t_n\tilde{u}_n)(t_n \tilde{u}_n)=\frac{1}{N}\int_{\mathbb{R}^N}|t_n\tilde{u}_n|^{1^*}\,dx \leq \frac{1}{N}\int_{\mathbb{R}^N}|\tilde{u}_n|^{1^*}\,dx,
$$
that is, 
$$
\frac{1}{N}S^N \leq \frac{1}{N}\int_{\mathbb{R}^N}|\tilde{u}_n|^{1^*}\,dx=\frac{1}{N}\int_{\Omega_{r_n}}|\tilde{u}_n|^{1^*}\,dx=I_\lambda(u_n)=J_{\lambda,\infty,r_n}.
$$
Hence, 
$$
\frac{1}{N}S^N \leq \limsup_{n \to +\infty}J_{\lambda,\infty,r_n} \leq \frac{1}{2N}S^N,
$$
which is absurd. 
\end{proof}

\begin{lemma}
For each  $1\leq k<\infty$ and $2\leq m<\infty$, we have that $J_{\lambda,k,r}<J_{\lambda,km,r},$ for all $r \geq r_0$.
\end{lemma}
\begin{proof}
	
Let $u \in \mathcal{N}_{km,r}$ be  such that $I_\lambda(u)=J_{\lambda,km,r}$ and fix $\varphi_n \subset C^{\infty}(\Omega) \cap  BV_{G_{km,r}}(\Omega_{r})$ such that
$$
\varphi_n \to u \quad \mbox{in} \quad L^{1}(\Omega_r) \quad \mbox{as} \quad n \to +\infty
$$
and
$$
\int_{\Omega_r}|\nabla \varphi_n|\,dx \to \int_{\Omega_r}|D u| \quad \mbox{as} \quad n \to +\infty.
$$
Associated with $\varphi_n$, we set $v_n(\theta,\rho,|y|)=\varphi_n(\theta/m,\rho,|y|)$ where $(\theta,\rho)$ is the polar coordinates of $x\in \mathbb{R}^2$  and $y\in \mathbb{R}^{N-2}.$ Hence, $v_n \in BV_{G_{k,r}}(\Omega_{r})$ and there is $t_n>0$ such that $\omega_n=t_nv_n\in \mathcal{N}_{k,r}.$ A simple argument proves that $(t_n)$ is bounded, then up to a  subsequence, $t_n \to t_0$ as $n \to +\infty$. Then
$$
\omega_n \to t_0u \quad \mbox{in} \quad L^{1}(\Omega_r)\quad \mbox{as} \quad n \to +\infty
$$
and
$$
\int_{\Omega_r}|\nabla \omega_n|\,dx \to t_0\int_{\Omega_r}|D u| \quad \mbox{as} \quad n \to +\infty.
$$
Thus,
$$
\begin{array}{l}
J_{\lambda,k,r}\leq I_{\lambda}(\omega_n)=\displaystyle \int_{\Omega_r}|\nabla \omega_n|\,dx\,dy+\int_{\Omega_r}|\omega_n|dxdy+\int_{\partial\Omega_r}|\omega_n|d \mathcal{H}^{N-1}\\
\mbox{}\\ 
\hspace{2.5 cm} -\displaystyle\int_{\Omega_r}\left(\frac{\lambda}{q}|\omega_n|^q+\frac{1}{1^*}|\omega_n|^{1^*}\right)\,dx\,dy.
\end{array}
$$	
Thereby,
$$
\begin{array}{l}
J_{\lambda,k,r}\leq \displaystyle \int_{0}^{\pi}\int_{r}^{r+1}\int_0^{2\pi}|\nabla \omega_n|\rho \,d\theta \,d\rho\,dy++\int_{\Omega_r}|\omega_n|dxdy+\int_{\partial\Omega_r}|\omega_n|\,d\mathcal{H}^{N-1}\\
\mbox{}\\
\hspace{1 cm} -\displaystyle  \int_{\Omega_r}\left(\frac{\lambda}{q}|\omega_n|^q+\frac{1}{1^*}|\omega_n|^{1^*}\right)\,dx\,dy,
\end{array}
$$
where 	$|\nabla \omega_n|=(\frac{1}{\rho^2m^2}(\omega_n)_\theta^2+(\omega_n)_\rho^2+|\nabla_y \omega_n|^2)^{\frac{1}{2}}.$ Using the fact that $m>1$ we get
$$
\begin{array}{l}
\displaystyle \int_{0}^{\pi}\int_{r}^{r+1}\int_0^{2\pi}\frac{1}{m^2\rho^2}(\omega_n)_\theta^2\rho\,d\theta\,d\rho\,dy<\int_{0}^{\pi}\int_{r}^{r+1}
\int_0^{2\pi}\frac{1}{\rho^2}(\omega_n)_\theta^2\rho\,d\theta\,d\rho\,dy+ \\
\mbox{}\\
\hspace{6.5 cm} \displaystyle +\left(\frac{1}{m^2}-1\right)\int_{0}^{\pi}\int_{r}^{r+1}\int_0^{2\pi}\frac{1}{\rho^2}(\omega_n)_\theta^2\rho\,d\theta\,d\rho\,dy.
\end{array}
$$
We claim $\displaystyle \liminf_{n \to +\infty}\int_{0}^{\pi}\int_{r}^{r+1}\int_0^{2\pi}\frac{1}{\rho^2}(\omega_n)_\theta^2\rho\,d\theta\,d\rho\,dy=\sigma	>0$, otherwise we have that for some subsequence
$$
\lim_{n \to +\infty}\int_{0}^{\pi}\int_{r}^{r+1}\int_0^{2\pi}(\omega_n)_\theta^2\,d\theta\,d\rho\,dy=0.
$$
Since $\omega_n \in W^{1,1}((0,\pi)\times (r,r+1)\times(0,2\pi))$ and $w_n \to t_0u$ in $L^{1}(\Omega)$, the last limit implies that $u(x)=u(|x|)$, that is $u \in BV_{G_{\infty,r}}(\Omega_r)$, which is absurd, since $J_{\lambda,km,r}<J_{\lambda,\infty,r}$,\ (see Lemmas \ref{Estimativa1} and \ref{Estimativa2}). The previous analysis ensures that
$$
J_{\lambda,k,r}\leq I_{\lambda}(t_0u)-\sigma < I_{\lambda}(t_0u) \leq I_{\lambda}(u)=J_{\lambda,km,r}.
$$
\end{proof}

\subsection{Proof of Theorem \ref{T1}}

Given $n \in \mathbb{N}$, by Lemmas \ref{Estimativa0} and \ref{Estimativa2} we know that $J_{\lambda,2^n,r}$ are critical levels of $I_\lambda$ with
$$
0<J_{\lambda,2,r}<J_{\lambda,2^2,r}<...<J_{\lambda,2^n,r}<J_{\lambda,\infty,r}.
$$
Applying the Principle of Symmetric Criticality (see \cite{Squassina}), it follows that they are critical points of $I_\lambda$ in $BV(\Omega_r)$. This way, all minimizers of $J_{\lambda,2^n,r}$ for $m = 1, . . . , n$ are nonradial, rotationally
non-equivalent and non-negative solutions of (\ref{P1}).

\section{Existence of multiple solutions via genus}

In this section we will prove Theorem \ref{T2}, which implies in the existence of multiple solutions for the problem
\begin{equation} \label{P2}
	\left\{	\begin{array}{l}
		-	\Delta_1 u =\lambda |u|^{q-2}u+|u|^{1^*-2}u, \quad\text{in }\Omega,\\
		u=0, \quad\text{on } \partial\Omega,
	\end{array}
	\right.
\end{equation}
where $\Omega \subset \mathbb{R}^N$ is a smooth bounded domain in $\mathbb{R}^N$ with $N\geq 2$, $\lambda>0$ and $q \in (1,1^*)$. Let us recall that $u \in BV(\Omega)$ is a solution of (\ref{P2}) if there is $z \in L^\infty(\Omega,\mathbb{R}^N)$ such that
\begin{equation}
	\left\{
	\begin{array}{l}
		|z|_\infty \leq 1, \quad \mbox{div} \, z \in L^{N}(\Omega), \quad -\displaystyle \int_{\Omega}u\mbox{div}z dx = \int_{\Omega}|Du|+\int_{\partial\Omega}|u|\,d \mathcal{H}^{N-1},\\
		\mbox{} \\
		-\mbox{div} \, z = \lambda |u|^{q-2}u+|u|^{1^*-2}u, \quad \mbox{a.e. in} \quad \Omega.
	\end{array}
	\right.
	\label{eulerlagrangeequationP2}
\end{equation}

In this section, we will consider the energy functional $I_\lambda : L^{1^*}(\Omega) \to (-\infty,+\infty]$,  given by
\begin{equation} \label{Ilambda2}
I_\lambda(u)=\int_{\Omega} |D u|\,dx+\int_{\partial\Omega}|u| \,d\mathcal{H}^{N-1}-\frac{\lambda}{q}\int_{\Omega}|u|^q\,dx-\frac{1}{1^*}\int_{\Omega}|u|^{1^*}\,dx.
\end{equation}

Hereafter, let us consider the functional $f_0:L^{1^*}(\Omega) \to [0,+\infty]$ given by
$$
f_0(u)=
\left\{
\begin{array}{l}
	\int_{\Omega}|D u|\,dx + \int_{\partial\Omega}|u| \,d\mathcal{H}^{N-1}, \quad \mbox{if} \quad u \in BV(\Omega) \\
	\mbox{} \\
	+\infty, \quad \mbox{if} \quad u \in L^{1^*}(\Omega) \setminus BV(\Omega),
\end{array}
\right.
$$
which is convex and lower semicontinuous in $L^{1^*}(\Omega)$. Moreover, let us define \linebreak $f_1: L^{1^*}(\Omega) \to [0,+\infty]$ by
$$
f_1(u) = \frac{\lambda}{q}\int_{\Omega}|u|^q\,dx + \frac{1}{1^*}\int_{\Omega}|u|^{1^*}\,dx,
$$
which is a $C^1$ functional. 

Then, the functional $I_\lambda$ is written as the difference between a convex, proper and lower semicontinuous functional and a $C^1$ one. Hence, in the light of \cite{Szulkin}, we denote by $\partial I_\lambda(u)$, the subgradient of $I_\lambda$ at $u \in L^{1^*}(\Omega)$, which is well defined as a subset of $L^N(\Omega)$.

By \cite[Proposition 4.23]{Kawohl}, we have the following result.

\begin{proposition}
\label{C1}
Assume that $u \in BV(\Omega)$ is a critical point of $I_\lambda$, i.e., $0 \in \partial I_\lambda(u)$. Then $u \in L^\infty(\Omega)$ and $u$ is a solution of \eqref{P2}, in the sense of \eqref{eulerlagrangeequationP2}.
\end{proposition}
\begin{proof}
Note that
$$
0 \in \partial I_\lambda(u)
$$
if and only if
$$
f_1'(u) \in \partial f_0(u).
$$
On the other hand, the last inclusion implies that there exists $w\in \partial f_{0}(u) \subset L^N(\Omega)$ such that
\begin{equation}
\label{eqsolution1}
f_1'(u) = w \quad \mbox{in $L^N(\Omega)$.}
\end{equation}
Taking into account the characterization of $\partial f_{0}(u)$  given in \cite[Proposition 4.23]{Kawohl}, there exists $z \in L^\infty(\Omega,\mathbb{R}^N)$, such that $|z|_\infty \leq 1$ and
\begin{equation}
	\left\{
	\begin{array}{l}
		-\mbox{div} \, z = w, \quad \mbox{a.e. in} \quad \Omega, \\
		\mbox{} \\
		\displaystyle \int_{\Omega}w u dx = \int_{\Omega}|Du|+\int_{\partial\Omega}|u|\,d \mathcal{H}^{N-1}. \\
		\end{array}
	\right.
	\label{eqsolution2}
\end{equation}
By \eqref{eqsolution1} and \eqref{eqsolution2}, we also have that
$$
-\mbox{div} \, z = \lambda |u|^{p-2}u + |u|^{1^*-2}u, \quad \mbox{a.e. in} \quad \Omega.
$$
Hence, $u$ satisfies \eqref{eulerlagrangeequationP2}. The fact that $u \in L^\infty(\Omega)$ is a regularity result which follows as in \cite[Proposition3.3]{DegiovanniMagrone}.

\end{proof}

Now let us define what we mean by a $(PS)$ sequence for $I_\lambda$. We say that $(u_k) \subset L^{1^*}(\Omega)$ is a $(PS)$ sequence for $I_\lambda$ if there exist $d \in \mathbb{R}$ and $(z_k) \subset L^N(\Omega)$ such that $|z_k|_N \to 0$ as $k \to +\infty$, 
\begin{equation}
\lambda|u_k|^{q-2}u_k + |u_k|^{1^*-2}u_k + z_k \in \partial f_0(u_k)
\label{PScondition}
\end{equation} 
and
$$
I_\lambda(u_k) \to d, \quad \mbox{as $k \to +\infty$}.
$$

Next, we state an abstract result, whose proof follows as in Szulkin  \cite[Theorem 4.4]{Szulkin}. Hereafter $X$ denotes a Banach space. We say that a functional $I : X \to (-\infty,+\infty]$ satisfies the condition $(H)$ if:

\begin{itemize}
\item [$(H)$] $I = \Phi+\psi$, where $\Phi \in C^{1}(X,\mathbb{R})$ and $\psi:X \to (-\infty,+\infty]$ is convex, proper (i.e. $\psi \not\equiv +\infty$) and lower semicontinuous.
\end{itemize}

Moreover, for each $c \in \mathbb{R}$ we denote
$$
I_{c}=\{u \in X\,:\,I(u) \leq c\}
$$
and by $\Sigma$ the collection of all symmetric subsets of $X\backslash\{0\}$ which are closed in $X$.

\begin{theorem}\label{AmRbthm}
	Assume that $I:X \to (-\infty,+\infty]$ satisfies $(H)$, $I(0)=0$ and $\Phi,\psi$ are even and there is $d>0$ such that $I$ has no critical points in $I_{-d}$. Assume also that
	\begin{enumerate}
		\item[a.] there is $M>0$ such that $I$ satisfies $(PS)_c$ condition for $0<c<M.$
		\item[b.] there exist $\alpha,\rho>0$ such that
		$$
		I(u)\geq\alpha \quad \mbox{for} \quad \,\,||u||=\rho.
		$$
		\item[c.] given $n \in \mathbb{N}$, there is a finite dimensional subspace $X_n \subset X$ and $R_n >\rho$ such that
		$$
		I|_{\partial Q_n}\leq -d \quad \mbox{where} \quad Q_n=\overline{B}_{R_n} \cap X_n.
		$$
	\end{enumerate}
Denoting by $\mathcal{F}$ the set
$$
\mathcal{F}=\{f \in C(Q_n,X)\,:\,f \,\,\mbox{is odd and } f|_{\partial Q_n} \approx id_{\partial Q_n} \,\, \mbox{in} \,\, I_{-d} \,\,\mbox{by an odd homotopy} \},
$$
we consider for each $j \in \mathbb{N}$ the sets $\Lambda'_j$ and $\Lambda_j$ given by,
$$
\Lambda'_j=
\left\{
\begin{array}{l}
	f(Q_n-V):\, f \in \mathcal{F}, V \, \mbox{is open in} \ Q_n \,\, \mbox{and symmetric,} \,\, V \cap \partial Q_n = \emptyset ,  \\
	\mbox{and for each } \,\, Y \subset V \mbox{such that} \,\, Y \in \Sigma, \gamma(Y) \leq k-j
\end{array}
\right\}
$$
and
$$
\Lambda_j=
\left\{
\begin{array}{l}
	A \subset X\,:\, A \,\, \mbox{is compact, symmetric and for each open set} \,\, U \supset A,  \\
	\mbox{there is} \,\, A_0 \in \Lambda'_j \,\, \mbox{such that} \,\, A_0 \subset U
\end{array}
\right\}.
$$
Using the above notation, the numbers
$$
c_j=\inf_{A \in \Lambda_j}\sup_{u \in A}I(u)
$$
are well defined for all $j \in \mathbb{N}$ and $0<\alpha\leq c_1\leq c_2\leq ....\leq c_{j}\leq c_{j+1}\leq ...$ for all $j \in\mathbb{N}$. If $c_n<M,$ then $c_j$ are critical values of $I$ for $j \in \{1,2,...,n\}$. Moreover, if there are $j_0 \in \{1,2,...,n\}$ and $p \in \mathbb{N}$ such that $c_{j_0}=...=c_{j_0+p}=c<M,$ then $\gamma(K_c)\geq p+1.$

\end{theorem}

Now, following the approach explored in \cite{DegiovanniMagrone}, for each $h>0$ we consider the functions $T_h,R_h:\mathbb{R} \to \mathbb{R}$ given by
$$
T_h(s) = \min\{ \max\{s,-h\},h\} \quad \mbox{and} \quad R_h(s)=s-T_h(s).
$$
A simple computation shows that for each $u \in L^{1^*}(\Omega)$, 
\begin{equation} \label{TH1}
T_h(u) \to u \quad \mbox{in} \quad L^{1^*}(\Omega) \quad \mbox{as} \quad h \to +\infty,
\end{equation}
and so,
\begin{equation} \label{RH1}
R_h(u) \to 0 \quad \mbox{in} \quad L^{1^*}(\Omega) \quad \mbox{as} \quad h \to +\infty.
\end{equation}
Moreover, if $(u_k) \subset L^{1^*}(\Omega)$ is a sequence satisfying 
$$
u_k(x) \to u(x) \quad \mbox{a.e. in} \quad \Omega \quad \mbox{as} \quad k \to +\infty,
$$
then, for each $h$ fixed, the Lebesgue Dominated Convergence Theorem ensures that 
\begin{equation}
T_h(u_k) \to T_h(u) \quad \mbox{in} \quad L^{1^*}(\Omega) \quad \mbox{as} \quad k \to +\infty.
\label{TH2}
\end{equation}

\begin{proposition} \label{Prop1} Let $(u_k)$ be a sequence in $BV(\Omega)$ and $(w_k)$ be a sequence in $L^{N}(\Omega)$ such that, for $k \in \mathbb{N}$, $w_k \in \partial f_0(u_k)$ and, as $k \to +\infty$,
$$
u_k \rightharpoonup u \quad \mbox{in $L^{1^*}(\Omega)$},
$$
$$
w_k \rightharpoonup w \quad \mbox{in $L^N(\Omega)$.}
$$
Then $u \in BV(\Omega)$ and  $w \in \partial f_0(u)$.
\end{proposition}
\begin{proof} See \cite[Proposition 3.2]{DegiovanniMagrone}.
	
\end{proof}

\begin{lemma} \label{ThRh} Let $(u_k)$ be a $(PS)$ sequence for $I_\lambda$. Assume that $(u_k)$ is bounded in $BV(\Omega)$ and 
$$
u_k \rightharpoonup u \quad \mbox{in $L^{1^*}(\Omega)$}.
$$
Then,
$$
\lim_{k \to +\infty}(f_0(u_k)-|u_k|_{1^*}^{1^*})=f_0(u)-|u|_{1^*}^{1^*} \leqno{(i)}
$$
and, for $h > 0$ fixed,
$$
\lim_{k \to +\infty}(f_0(R_h(u_k))-|R_h(u_k)|_{1^*}^{1^*}) \leq f_0(R_h(u))-|R_h(u)|_{1^*}^{1^*}. \leqno{(ii)}
$$
\end{lemma}
\begin{proof} Since $(u_k)$ is a $(PS)$ sequence, there is $(z_k) \subset L^{N}(\Omega)$, where $z_k = o_k(1)$ in $L^{N}(\Omega)$ and
\begin{eqnarray*}
f_0(v)-f_0(u_k) & \geq & \lambda\int_{\Omega}|u_k|^{q-2}u_k(v-u_k)\,dx+\int_{\Omega}|u_k|^{1^*-2}u_k(v-u_k)\,dx\\
& & +\int_{\Omega}z_k(v-u_k)\,dx,
\end{eqnarray*}
for all $v \in L^{1^*}(\Omega)$. From this, 
$$
\lambda |u_k|^{q-2}u_k+|u_k|^{1^*-2}u_k+z_k \in \partial f_0(u_k), \quad \forall k \in\mathbb{N},
$$
and then there exists $w_k \in \partial f_0(u_k) $ such that
\begin{equation}
w_k=\lambda |u_k|^{q-2}u_k+|u_k|^{1^*-2}u_k+z_k, \quad \forall k \in \mathbb{N}.
\label{eqprop1}
\end{equation}
Moreover, by \cite[Proposition 4.23]{Kawohl},
\begin{equation}
\int_{\Omega}w_ku_k\,dx=	\int_{\Omega}|D u_k|\,dx +\int_{\partial\Omega}|u_k| \,d\mathcal{H}^{N-1}, \quad \forall k \in \mathbb{N}.
\label{eqprop2}
\end{equation}
Thus, from \eqref{eqprop1} and \eqref{eqprop2}, for all $k \in \mathbb{N}$,
\begin{eqnarray}
\nonumber \int_{\Omega}|D u_k|\,dx +\int_{\partial\Omega}|u_k| \,d\mathcal{H}^{N-1} & = & \lambda \int_{\Omega}|u_k|^q\,dx+\int_{\Omega}|u_k|^{1^*}\,dx\\
\label{eqprop11}& & +\int_{\Omega}z_ku_k\,dx.
\end{eqnarray}

Since $q \in (1,1^*)$, H\"older's inequality implies that
$\left(|u_k|^{q-2}u_k\right)$ and $\left(|u_k|^{1^*-2}u_k\right)$ are bounded in $L^N(\Omega)$. Indeed, it is straightforward to see that 
$$
\left||u_k|^{1^*-2}u_k\right|_N = |u_k|_{1^*}^\frac{1}{N-1}
$$
and
\begin{eqnarray*}
\left||u_k|^{q-2}u_k\right|_N^N & = & \int_\Omega |u_k|^{(q-1)N} dx\\
&  \leq & \left( \int_\Omega |u_k|^{1^*} dx\right)^{(N-1)(q-1)} |\Omega|^{1-(N-1)(q-1)},
\end{eqnarray*}
from where it follows that both these sequences are bounded in $L^N(\Omega)$.
Then
\begin{equation}
\label{eqprop31}
|u_k|^{q-2}u_k \rightharpoonup |u|^{q-2}u
\end{equation}
and
\begin{equation}
\label{eqprop32}
|u_k|^{1^*-2}u_k \rightharpoonup |u|^{1^*-2}u
\end{equation}
in $L^N(\Omega)$.

Hence, from \eqref{eqprop31} and  \eqref{eqprop32}, 
\begin{equation}
\label{eqprop4}
w_k \rightharpoonup w \quad \mbox{in $L^N(\Omega)$,}
\end{equation}
with
$$
w = -\gamma + \lambda |u|^{q-2}u+|u|^{1^*-2}u.
$$

Taking into account the hypothesis and \eqref{eqprop4}, Proposition \ref{Prop1} yields that $u \in BV(\Omega)$, 
$$
\lambda |u|^{q-2}u+|u|^{1^*-2}u \in \partial f_0(u)
$$
and then, by (\ref{eqprop2}), 
\begin{equation}
\int_{\Omega}|D u|\,dx +\int_{\partial\Omega}|u| \,d\mathcal{H}^{N-1}= \lambda \int_{\Omega}|u|^q\,dx+\int_{\Omega}|u|^{1^*}\,dx.
\label{eqprop41}
\end{equation}
Hence, from \eqref{eqprop11},
$$
\begin{array}{l}
\displaystyle \lim_{k \to +\infty}(f_0(u_k)-|u_k|_{1^*}^{1^*})=\lim_{k \to +\infty}\left( \lambda \int_{\Omega}|u_k|^q\,dx+\int_{\Omega}z_ku_k\,dx\right) \\
\mbox{} \\
\hspace{4 cm} = \displaystyle \lambda \int_{\Omega}|u|^q\,dx.
\end{array}
$$
Then, from the last equality and \eqref{eqprop41}, it follows that
$$
\lim_{k \to +\infty}(f_0(u_k)-|u_k|_{1^*}^{1^*})=(f_0(u)-|u|_{1^*}^{1^*}),
$$
showing $(i)$. The item $(ii)$ follows as in \cite[Lemma 5.1]{DegiovanniMagrone}.

\end{proof}

\begin{lemma} Each $(PS)$ sequence for $I_\lambda$ is bounded in $BV(\Omega)$.
	
\end{lemma}
\begin{proof} Let $(u_k)$ be a $(PS)_d$ sequence for $I_\lambda$, that is,
$$
I_\lambda(u_k) \to d \quad \mbox{as} \quad k \to +\infty
$$	
and
\begin{eqnarray*}
f_0(v)-f_0(u_k) & \geq & \lambda\int_{\Omega}|u_k|^{q-2}u_k(v-u_k)\,dx\\
& & +\int_{\Omega}|u_k|^{1^*-2}u_k(v-u_k)\,dx + \int_{\Omega}z_k(v-u_k)\,dx,
\end{eqnarray*}
where $(z_k) \subset L^{N}(\Omega)$, with $z_k = o_k(1)$ in $L^{N}(\Omega)$, as $k \to +\infty$.

By Proposition \cite[Proposition 4.23]{Kawohl}, for all $k \in \mathbb{N}$,
\begin{eqnarray*}
\int_{\Omega}|D u_k|\,dx +\int_{\partial\Omega}|u_k| \,d\mathcal{H}^{N-1} & = & \lambda \int_{\Omega}|u_k|^q\,dx + \int_{\Omega}|u_k|^{1^*}\,dx \\
& & +\int_{\Omega}z_k u_k\,dx.
\end{eqnarray*}
Now, let us denote
\begin{eqnarray*}
\mathcal{Q}(u_k) & = & \int_{\Omega}|D u_k|\,dx +\int_{\Omega}|u_k|\,dx+\int_{\partial\Omega}|u_k| \,d\mathcal{H}^{N-1} \\
& & -\lambda \int_{\Omega}|u_k|^q\,dx - \int_{\Omega}|u_k|^{1^*}\,dx-\int_{\Omega}z_ku_k\,dx
\end{eqnarray*}
and note that
\begin{equation}
\mathcal{Q}(u_k)=0, \quad \forall k \in \mathbb{N}.
\label{eqPSbounded2}
\end{equation}
Thus, from \eqref{eqPSbounded2}, 
\begin{eqnarray*}
d+o_k(1) & = & I_\lambda(u_k)\\
& = & I_\lambda(u_k)-\frac{1}{q}\mathcal{Q}(u_k)\\
&  \geq & \left(1-\frac{1}{q}\right)f_0(u_k)+\left(1-\frac{1}{q}\right)|u_k|_1+\left(\frac{1}{q}-\frac{1}{1^*}\right)|u_k|_{1^*}^{1^*} + \frac{1}{q}\displaystyle \int_{\Omega}z_ku_k\,dx \\
& \geq & \left(1-\frac{1}{q}\right)f_0(u_k) + \left(\frac{1}{q}-\frac{1}{1^*}\right)|u_k|_{1^*}^{1^*}-\frac{1}{q}|z_k|_N|u_k|_{1^*} \\
& \geq  & \left(1-\frac{1}{q}\right)f_0(u_k) + \left(\frac{1}{q}-\frac{1}{1^*}\right)(|u_k|_{1^*}^{1^*}-|u_k|_{1^*}),
\end{eqnarray*}
for $k$ large enough. Since $g:[0,+\infty) \to \mathbb{R}$, given by
$$
g(t)=t^{1^*}-t
$$
is bounded from below, there exists $K > 0$ such that
$$
g(t) \geq -K, \quad \forall t \in [0,+\infty).
$$
Then
$$
d+o_k(1) \geq \left(1-\frac{1}{q}\right)f_0(u_k) - \left(\frac{1}{q}-\frac{1}{1^*}\right)K,
$$
from where it follows that $(u_k)$ is bounded in $BV(\Omega)$.
\end{proof}

\begin{lemma} 
\label{Pscondition}
For each $\lambda > 0$, the functional $I_\lambda$ satisfies the $(PS)_c$ condition, for $c<\frac{1}{N}S^{N}$.	
\end{lemma}
\begin{proof} Let $(u_k)$ be a $(PS)_c$ sequence for $I_\lambda$ with $c < \frac{1}{N}S^{N}$. Then,
$$
I_\lambda(u_k) \to c \quad \mbox{as} \quad k \to +\infty
$$	
and
\begin{eqnarray*}
f_0(v) - f_0(u_k) & \geq & \lambda\int_{\Omega}|u_k|^{q-2}u_k(v-u_k)\,dx\\
& & +\int_{\Omega}|u_k|^{1^*-2}u_k(v-u_k)\,dx + \int_{\Omega}w_k(v-u_k)\,dx,
\end{eqnarray*}
for some $w_k \in L^{N}(\Omega)$ and $w_k = o_k(1)$ in  $L^{N}(\Omega)$. Moreover, we also have
\begin{eqnarray}
\label{eqPS1}\int_{\Omega}|D u_k|\,dx +\int_{\partial\Omega}|u_k| \,d\mathcal{H}^{N-1} & = & \lambda \int_{\Omega}|u_k|^q\,dx+\int_{\Omega}|u_k|^{1^*}\,dx\\
\nonumber& & +\int_{\Omega}w_ku_k\,dx, \quad \forall k \in \mathbb{N}.
\end{eqnarray}
Hence, from \eqref{eqPS1}, for all $k \in \mathbb{N}$,
\begin{equation}
I_\lambda(u_k)=\lambda\left(1-\frac{1}{q}\right)|u_k|_{q}^{q}+\left(1-\frac{1}{1^*}\right)|u_k|_{1^*}^{1^*}+\int_{\Omega}w_ku_k\,dx.
\label{eqPS2}
\end{equation}
Since $(u_k)$ is bounded and $w_k = o_k(1)$ in  $L^{N}(\Omega)$, \eqref{eqPS2} gives
$$
\lim_{k \to +\infty}|u_k|_{1^*}^{1^*} \leq Nc < S^N. 
$$

Since $f_0$ is lower semicontinuous and
$$
f_0(u)=f_0(T_h(u))+f_0(R_h(u)),
$$
(\ref{TH1}) and (\ref{RH1}) lead us to
$$
\limsup_{h \to +\infty}(f_0(R_h(u))-|R_h(u)|_{1^*}^{1^*}) \leq 0.
$$
Therefore, given $\epsilon>0$, there is $h>0$ large enough such that
\begin{equation} \label{S2EQ1}
f_0(R_h(u))-|R_h(u)|_{1^*}^{1^*}<\epsilon\left(S-(Nc)^{\frac{1}{N}} \right).
\end{equation}
For $h$ fixed above, the definition of $R_h$ gives
$$
\limsup_{k \to +\infty}|R_h(u_k)|_{1^*}^{1^*-1}\leq \limsup_{k \to +\infty}|u_k|_{1^*}^{1^*-1} \leq (Nc)^{\frac{1}{N}}.
$$
Now, the inequality below
$$
\left(S-|R_h(u_k)|_{1^*}^{1^*-1}\right)|R_h(u_k)|_{1^*}\leq f_0(R_h(u_k))-|R_h(u_k)|_{1^*}^{1^*}
$$
together with Lemma \ref{ThRh} and (\ref{S2EQ1}) leads to
$$
\limsup_{k \to +\infty}|R_h(u_k)|_{1^*}<\epsilon.
$$
Hence $|R_h(u)|_{1^*}<\epsilon$. Moreover, since by \eqref{TH2}, $(T_h(u_k))$ is strongly convergent to $T_h(u)$, it follows that
\begin{eqnarray*}
\limsup_{k \to +\infty}|u_k-u|_{1^*} & \leq & \limsup_{k \to +\infty}|T_h(u_k)-T_h(u)|_{1^*} \\ 
& & + \limsup_{k \to +\infty}|R_h(u_k)|_{1^*}+|R_h(u)|_{1^*}\\
& \leq & 2 \epsilon.
\end{eqnarray*}
Since that $\epsilon$ is arbitrary, the last inequality ensures that $u_k \to u$ in $L^{1^*}(\Omega)$.
\end{proof}
\begin{lemma} \label{mpg} There are $\alpha,\rho>0$ such that
$$
I_\lambda(u) \geq \alpha, \quad \mbox{for} \quad |u|_{1^*}=\rho.
$$	
\end{lemma}
\begin{proof} Note that, in order to verify this lemma, it suffices to consider $u \in BV(\Omega)$, since otherwise we would have $I_\lambda(u)=+\infty$. Then, if $u \in BV(\Omega)$, from the continuous embedding  $BV(\Omega) \hookrightarrow L^{1^*}(\Omega)$ and H\"older inequality, we have that
$$
I_\lambda(u) \geq C_1|u|_{1^*}-C_2|u|_{1^*}^{q}-|u|_{1^*}^{1^*}.
$$
Since $q > 1$, the last inequality allows us to conclude that there are $\alpha,\rho>0$ such that
$$
I_\lambda(u) \geq \alpha, \quad \mbox{for} \quad |u|_{1^*}=\rho.
$$
\end{proof}

\subsection{Proof of Theorem \ref{T2}}
In what follows, we will assume that there is $d >0$ such that $I_\lambda$ has no critical point in $I_{-d}$, otherwise $I_\lambda$ has infinitly many critical points and Theorem \ref{T2} is proved.

\begin{lemma} \label{terceirag} For $n \in \mathbb{N}$ and a finite dimensional subspace $X_n \subset X$, there exists $R_n >\rho$ such that
$$
	I_\lambda|_{\partial Q_n}\leq -d \quad \mbox{where} \quad Q_n=\overline{B}_{R_n} \cap X_n.
$$
\end{lemma}
\begin{proof} Let $X_n \subset L^{1^*}(\Omega)$ be a finite dimensional subspace, such that $X_n \subset C_0^{\infty}(\Omega)$. Since in $X_n$, all the norms are equivalent, there are positive constants $a_n, d_n$ and $b_n$ (which depend just on $n \in \mathbb{N}$), such that, for $u \in X_n$, 
$$
I_\lambda(u) \leq a_n|u|_{1^*}-d_n\lambda|u|_{1^*}^q-b_n|u|_{1^*}^{1^*}.
$$
The last inequality, in turn, implies that
$$
I_\lambda(u) \to -\infty \quad \mbox{as} \quad |u|_{1^*} \to +\infty.
$$
This proves the desired result. 	
\end{proof}
\begin{lemma} \label{estimacn}  For each $n \in \mathbb{N}$, there is $\lambda_n>0$ such that if $\lambda \geq \lambda_n$, then
$$
\sup_{u \in Q_n}I_\lambda(u) < \frac{1}{N}S^N.
$$
Hence, $c_n < \frac{1}{N}S^N$.
\end{lemma}	
\begin{proof} Arguing as in the proof of Lemma \ref{terceirag}, we get
$$
\sup_{u \in Q_n}I_\lambda(u)\leq \sup_{u \in Q_n}\left\{a_n|u|_{1^*}-d_n\lambda|u|_{1^*}^q-b_n|u|_{1^*}^{1^*}\right \} \leq \sup_{u \in Q_n}\left\{a_n|u|_{1^*}-d_n\lambda|u|_{1^*}^q\right \}
$$
Defining the function $h:[0,+\infty)	\to \mathbb{R}$ as
$$
h(t)=a_nt-d_n\lambda t^q,
$$
it is straightforward to see that
$$
\max_{t \geq 0}h(t)=C_n\left(\frac{1}{\lambda}\right)^{\frac{1}{q-1}},
$$
for some $C_n$ which depends on $n$. Thus, there is $\lambda_n>0$ such that
$$
\max_{t \geq 0}h(t) < \frac{1}{N}S^N, \quad \forall \lambda \geq \lambda_n.
$$
This ensures that
$$
\sup_{u \in Q_n}I_\lambda(u) < \frac{1}{N}S^N, \quad \forall \lambda \geq \lambda_n.
$$
Since $Q_n \in \Lambda_n$, we have that
$$
c_n \leq \sup_{u \in Q_n}I_\lambda(u) < \frac{1}{N}S^N, \quad \forall \lambda \geq \lambda_n.
$$
\end{proof}

Therefore, taking into account Lemmas \ref{Pscondition}, \ref{mpg}, \ref{terceirag} and \ref{estimacn}, we see that $I_\lambda$ satisfies all the conditions of Theorem \ref{AmRbthm}, and so, Theorem \ref{T2} is proved.

\end{document}